\documentclass[a4paper,11pt,final]{article}
\usepackage{amsmath, amsthm, amssymb}
\usepackage{latexsym}
\usepackage[latin1]{inputenc}
\usepackage{epsf,exscale,times}
\usepackage[T1]{fontenc}
\usepackage[dvips]{graphicx}
\usepackage{graphics,epsfig}
\usepackage{epsfig,rotate,color}
\usepackage{showkeys}
\usepackage{subfigure}  

\newcommand{\SR}{\mathcal{S}(\mathbb{R})}
\newcommand{\LdeuxR}{L^{2}(\mathbb{R})}

\newcommand{\I}{\mathcal{I}}

\newcommand{\F}{\mathcal{F}}
\newcommand{\R}{\mathbb{R}}

\newcommand{\n}{\vert \vert}
\newtheorem{definition}{Definition}

\newtheorem{remark}{Remark}
\newtheorem{lemma}{Lemma}
\newtheorem{theorem}{Theorem}

\title{On the instability of a nonlocal conservation law}
      
\author{ Afaf Bouharguane \footnote{Institut de Math\'ematiques et Mod\'elisation de Montpellier, UMR 5149 CNRS, Universit\'e Montpellier 2, Place Eug\`ene Bataillon, CC 051
34095 Montpellier, France. Email:\, {\sffamily bouharg@math.univ-montp2.fr}
} 
}
\date{\today}

\begin{document}
\maketitle

\begin{abstract}
We are interested in a nonlocal conservation law which describes the morphodynamics of sand dunes sheared by a fluid flow, recently proposed by Andrew C. Fowler and studied by \cite{alibaud,alvarez}. We prove that constant solutions of Fowler's equation are non-linearly unstable. We also illustrate this fact using a finite difference scheme.\\
\end{abstract}



\section{Introduction}

Partial Differential Equations with nonlocal or fractional operators are widely used to model scientific problems in mechanics, physics, signal processing and other subjects, see for example \cite{signal} and references therein. We consider in this paper a nonlocal conservation law which appears in the formation and dynamics of sand structures such as dunes and ripples \cite{fowler2,lagree}. Namely, Fowler (\cite{fowler2,fowler3,fowler1}) introduced the following equation \label{fowler1}
\begin{equation}
\begin{cases}
\partial_t u(t,x) + \partial_x\left(\frac{u^2}{2}\right) (t,x) + \I [u(t,\cdot)] (x)  -
\partial^2_{xx} u(t,x) = 0 & t \in (0,T), x \in \R, \\
u(0,x)=u_{0}(x) & x \in \R,
\end{cases}
\label{fowlereqn}
\end{equation}
where $u=u(t,x)$ represents the dune height, $u_0 \in L^{2}(\R)$ is an initial condition, $T$ is any given positive time and $\I$
is a nonlocal operator defined as follows: for any
Schwartz function $\varphi \in \SR$ and any $x \in \mathbb{R}$,
\begin{equation}
\I [\varphi] (x) := \int_{0}^{+\infty} |\zeta|^{-\frac{1}{3}}
\varphi''(x-\zeta) d\zeta . \label{nonlocalterm}
\end{equation}
Equation \eqref{fowlereqn} is valid for a river flow over an erodible bottom $u(t,x)$ with slow variation. The nonlocal term appears after a subtle modeling of the basal shear stress. 
This operator appears also in the work of Kouakou \& Lagr\'ee \cite{kouakou,lagree}. \\
The nonlocal term $\I$ can be seen as a fractional power of order $2/3$ of the Laplacian with the bad sign. Indeed, it has been proved \cite{alibaud} 
\begin{equation} 
 \F \left( \I [\varphi]-\varphi''\right) (\xi) = \psi_{\mathcal{I}}(\xi)  \F
\varphi (\xi)
\label{fourier}
\end{equation}
where 
\begin{equation}
\psi_{\mathcal{I}}(\xi)=4 \pi^2 \xi^2-a_{\I}
|\xi|^{\frac{4}{3}} + i \: b_{\I} \xi |\xi|^{\frac{1}{3}},
\label{psiexpression}
\end{equation}
 with $a_\I, b_\I$ positive constants
and $\F$ denotes the Fourier transform. One simple way to establish this fact is the derivation of a new formula for the operator $\I$ \cite{alibaud}
\begin{equation}
\I [\varphi](x)= C_{\I} \int_{-\infty}^{0}
\frac{\varphi(x+z)-\varphi(x)-\varphi'(x) z}{|z|^{7/3}} \:
dz,\label{intformula}
\end{equation}
with $C_\I=\frac{4}{9}$. \\ 
The operator $\I[u]$ is a weighted mean of second derivatives of $u$ with the bad sign and has therefore an anti-diffusive effect and creates instabilities which are controlled by the diffusive operator $-\partial_{xx}^2$. This remarkable feature enabled to apply this model for signal processing. Indeed, the diffusion is used to reduce the noise whereas the nonlocal anti-diffusion  is used to enhance the contrast \cite{signal}. \\

\begin{remark} 
For causal functions (i.e. $\varphi(x) = 0 $ for $x<0$), this operator is 
 up to multiplicative constant, the Riemann-Liouville fractional derivative operator which is defined as follows \cite{pod}
 \begin{equation}
\frac{1}{\Gamma(2/3)} \int_0^{+\infty} \frac{\varphi^{''}(x-\xi)}{|\xi|^{1/3}} d\xi= \frac{d^{-2/3}}{d x^{-2/3}} \varphi'' (x) = 
\frac{d^{4/3}}{d x^{4/3}} \varphi (x),
\end{equation}
where $\Gamma$ denotes the Euler function.
\end{remark}

Recently, some results regarding this equation have been obtained, namely, existence of travelling-waves $u_{\phi}(t,x) = \phi(x - ct)$ where $\phi \in C^1_{b}(\R)$ and $c \in \R$ represents wave velocity, the global well-posedness for $L^2$-initial data, the failure of the maximum principle, the local-in-time well-posedness in a subspace of $C^1_{b}$ \cite{alibaud, alvarez} and
the global well-posedness in a $L^2$-neighbourhood of $C^1_{b}$, solving namely for  $u = u_{\phi} +v$, with $v \in L^2(\R)$ \cite{afaf}.
For this purpose, the following Cauchy problem solved by the perturbation $v$ has been considered:
\begin{equation}
\begin{cases}
\partial_t v(t,x) + \partial_x(\frac{v^2}{2}+ u_{\phi} v)(t,x) + \I[v(t, \cdot)](x) -
\partial^2_{xx} v(t,x) = 0 & t \in (0,T), x \in \R, \\
v(0,x)=v_{0}(x) & x \in \R,
\end{cases}
\label{fowlermodif}
\end{equation}
where $v_0 \in L^{2}(\R)$ is an initial perturbation and $T$ is any given positive time. \\

Let us note also that any constant is solution of the Fowler equation. We shall prove below that these solutions are unstable. \\
To prove that a solution $u_\phi$ is unstable when $\phi$ is constant, we introduce the notion of \emph{mild solution} (see Definition \ref{mild}) based on Duhamel's formula \eqref{duhamel}. We also give some numerical results that illustrate this fact. \\

The remaining of this paper is organized as follows: in the next section, we define the notion of mild solution and we give some results.
Section \ref{nostability} contains the proof of the instability. We introduce in section \ref{numerik} an explicit finite difference scheme for which we give numerical simulations to illustrate the theory of the previous section. 

\section{Duhamel formula and some results}

\begin{definition}
\label{mild}
Let $u_\phi$ be a constant, $T>0$ and $v_{0} \in L^{2}(\R)$. We say that $v \in L^\infty((0,T);L^2(\R)) $ is a \textbf{mild solution} to \eqref{fowlermodif} if for any $t \in \left(0,T\right)$,
\begin{equation}
v(t,x)=K(t, \cdot)\ast v_{0}(x)- \int^{t}_{0} \partial_{x}K(t-s,\cdot)\ast \left(\frac{v^{2}}{2}\right)(s,\cdot)(x) \: ds
\label{duhamel}
\end{equation}
where $K(t,x)=\F^{-1}\left(e^{-t \phi_{\I}(\cdot)}\right)(x)$
is the kernel of the operator $\I-\partial^2_{xx} + u_\phi \partial_x$ and $\phi_\I(\xi) = 4 \pi^2 \xi^2-a_{\I}
|\xi|^{\frac{4}{3}} + i \: b_{\I} \xi |\xi|^{\frac{1}{3}} + i\: 2 \pi u_\phi \xi $.
\end{definition}

The expression \eqref{duhamel} is the Duhamel formula and is obtained using the spatial Fourier transform.                                                                           The use of this formula allows to prove the local-in-time existence with the help of the contracting fixed point theorem. The global existence is obtained thanks to $L^2$ a priori estimate. We refer to \cite{alibaud,afaf} for the proof. \\

\begin{lemma}\label{kernel2}
Let $t>0$. Then, $\partial_x K(t,\cdot) \in L^2(\R)$ and satisfies
\begin{equation}
||\partial_x K(t, \cdot) ||_{L^2(\R)} \leq C\left(t^{-3/4} + e^{\alpha t}\right), 
\end{equation}
where $C$ is a positive constant independent of $t$. 
\end{lemma}

\begin{proof}
By Plancherel formula, we have
\begin{eqnarray*}
||\partial_x K(t, \cdot)||^2_{L^2(\R)} &=&  ||\F\left( \partial_x K(t, \cdot)\right) ||^2_{L^2(\R)} = ||\xi \mapsto 2 i \pi \xi e^{-t\phi_\I(\xi)} ||^2_{L^2(\R)}, \\
&=& 2 \int_0^{+\infty} 4 \pi^2 \xi^2 e^{-2t(4\pi^2 \xi^2- a_\I \xi^{4/3})} \, d\xi.
\end{eqnarray*}
Let $\xi_0>0$ such that for all $\xi>\xi_0 $, 
$$ 4\pi^2 \xi^2- a_\I \xi^{4/3} \geq \xi^2.$$
Then, if we denote by $\alpha = - \min Re(\phi_\I)$ (see Figure \ref{psifigure}), we have
\begin{eqnarray*}
||\partial_x K(t, \cdot)||^2_{L^2(\R)}
&\leq & 8 \pi^2  \int_0^{\xi_0} \xi^2 e^{-2t(4\pi^2 \xi^2- a_\I \xi^{4/3})} \, d\xi +  8 \pi^2 \int_{\xi_0}^{+\infty} \xi^2 e^{-2t(4\pi^2 \xi^2- a_\I \xi^{4/3})} \, d\xi, \\
&\leq & 8 \pi^2   \int_0^{\xi_0} \xi^2 e^{2t \alpha} \, d\xi + 8 \pi^2  \int_{0}^{+\infty} \xi^2 e^{-2t\xi^2} \, d\xi, \\
&\leq & 8 \pi^2  \frac{\xi_0^3}{3} e^{2t \alpha}  + \frac{4\pi^2 }{\sqrt{2}} t^{-3/2} \int_{0}^{+\infty} \xi^2 e^{-\xi^2} \, d\xi, \\
&\leq & C \left(t^{-3/2} + e^{2\alpha t} \right), 
\end{eqnarray*}
which completes the proof of this lemma.
\end{proof}

\begin{remark}
\label{remarkL2}
Using again Plancherel formula, we have for any initial data $v_0$ with values in $L^2(\R)$, the following $L^2$-estimate \cite{alibaud}:
\begin{equation*}
||v(t, \cdot)||_{L^2(\R)} \leq e^{\alpha t} ||v_0||_{L^2(\R)}.
\end{equation*}
\end{remark}

\section{Instability\label{nostability}}

In this section, we give the proof of the instability.  \\ 

Let us first note that a fundamental property of the kernel $K$ is the non-positivity \cite{alibaud}, see Figure \ref{kernelfig}. This feature enabled to prove the failure of the maximum principle. We use again this property to show that the constant solutions of the Fowler equation are unstable. 

\begin{figure}[htp]
\begin{center}
  \includegraphics[width=9cm,height=5cm]{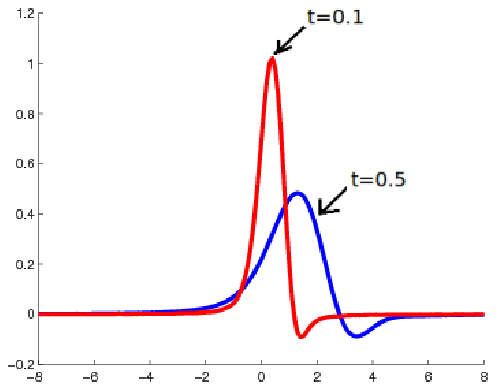}
  \caption{Evolution of the kernel $K$ for $t=0.1$ and $t=0.5$ s }\label{kernelfig}
  \end{center}
\end{figure}

 
\begin{theorem}
Any constant solution $u_\phi$ of the Fowler equation is unstable, i.e. 
$\exists \varepsilon >0, \forall \delta >0, \exists v_0 \in L^2(\R)$ \mbox{ an initial data with } $||v_0||_{L^2(\R)} \leq \delta$ \mbox{ and } $t_0>0$ \mbox{ such that } $||v(t_0, \cdot)||_{L^2(\R)} > \varepsilon$, 
\label{theoreme2}
where $v(t, \cdot)$ is the solution of \eqref{fowlermodif} with $v(0,\cdot)=v_0$. 
\end{theorem}


\begin{proof}
We adapt the strategy used by De Bouard in \cite{anne}, making it more elementary in our setting. \\
We denote by $S(t)$ the linear semigroup associated with equation \eqref{fowlermodif} i.e. $S(t)w = K(t, \cdot)\ast w$. Hence, we have
\begin{equation}
T(t)v_0 := v(t, \cdot) = S(t)v_0 - \frac{1}{2} \int_0^t \partial_x K(t-s, \cdot) \ast v^2(s, \cdot) \, ds,  
\end{equation}
where $T$ denotes the nonlinear semigroup associated with equation \eqref{fowlermodif}.\\
First, by Young inequality, Lemma \ref{kernel2} and Remark \ref{remarkL2}, we have
\begin{eqnarray}
||T(t)v_0-S(t)v_0 ||_{L^2(\R)} &\leq& \frac{1}{2} \int_0^t ||\partial_x K(t-s, \cdot)||_{L^2(\R)} ||v(s, \cdot)||_{L^2(\R)}^2 \, ds, \nonumber \\
&\leq & \frac{C}{2} \int_0^t \left[ (t-s)^{-3/4} + e^{\alpha(t-s)} \right] e^{2\alpha s} ||v_0||_{L^2(\R)}^2 \, ds, \nonumber \\ 
&\leq & b(t) ||v_0||_{L^2(\R)}^2,
\end{eqnarray}
where $b(t) = \frac{C}{2} e^{2\alpha t} (4 t^{1/4} + \frac{1}{\alpha}e^{\alpha t})>0$ for all $t\in (0,T)$. \\
Hence, we obtain for $t_0>0$ fixed,
\begin{equation}
||T(t_0)v_0-S(t_0)v_0 || \leq  b_0 ||v_0||_{L^2(\R)}^2,
\label{difft0}
\end{equation}
where $b_0=b(t_0)$. \\
Assume now that the constant $u_\phi$ is a stable solution of the Fowler equation i.e, 
$$\forall \varepsilon >0, \exists \eta >0, ||v_0|| < \eta \Rightarrow ||v(t,\cdot)||_{L^2(\R)} < \varepsilon, $$ 
for all positive $t$. 
Let then 
\begin{equation*}
0<\varepsilon < \frac{e^{\alpha t_0} - 1}{16 b_0},
\end{equation*}
and let $\eta>0$ be such that $||v_0||_{L^2(\R)} < \eta \Rightarrow ||v(t, \cdot)|| \leq \varepsilon$, for all positive $t$. \\
Let $N$ be an integer large enough such that
\begin{equation}
e^{\alpha t_0 N} \geq \frac{e^{\alpha t_0} - 1 }{4 \eta b_0}.
\label{Nchoisi}
\end{equation}
Let us now consider the following initial data: $v_0 = \delta w_0$, where $\F(w_0) = \frac{1}{\sqrt{d-c}} 1_{[c,d]}$,  with  $0<c<d$ satisfying (see Figure \ref{psifigure}): 
\begin{itemize}
\item $\mbox{Re}\left( \phi_{\I}\right)(c)<\mbox{Re}\left( \phi_{\I}\right) (d)=-\beta<0$,
\item for all $\xi \in [c,d], \mbox{Re}\left( \phi_{\I}\right)(\xi) \leq - \beta$,
\item $\beta = \alpha - \gamma$ where $ 0< \gamma < \frac{\ln(2)}{Nt_0}$.
\end{itemize}
And $\delta$ is defined as
\begin{equation}
\delta  = e^{-\alpha N t_0} \frac{e^{\alpha t_0} - 1 }{4 b_0}.
\label{definitiondelta}
\end{equation}
The shape of $w_0$ is given by
\begin{equation*}
w_{0}(x) = \frac{1}{\sqrt{d-c}} \left\{
    \begin{array}{ll}
         2\frac{\sin(\frac{d-c}{2}x)}{x}e^{-i \frac{d+c}{2} x}  & \mbox{if } x \neq 0, \\
         d-c & \mbox{otherwise.}
    \end{array}
\right.
\end{equation*}
\begin{figure}[h!]
	\centering
	\includegraphics[width=8cm,height=55mm]{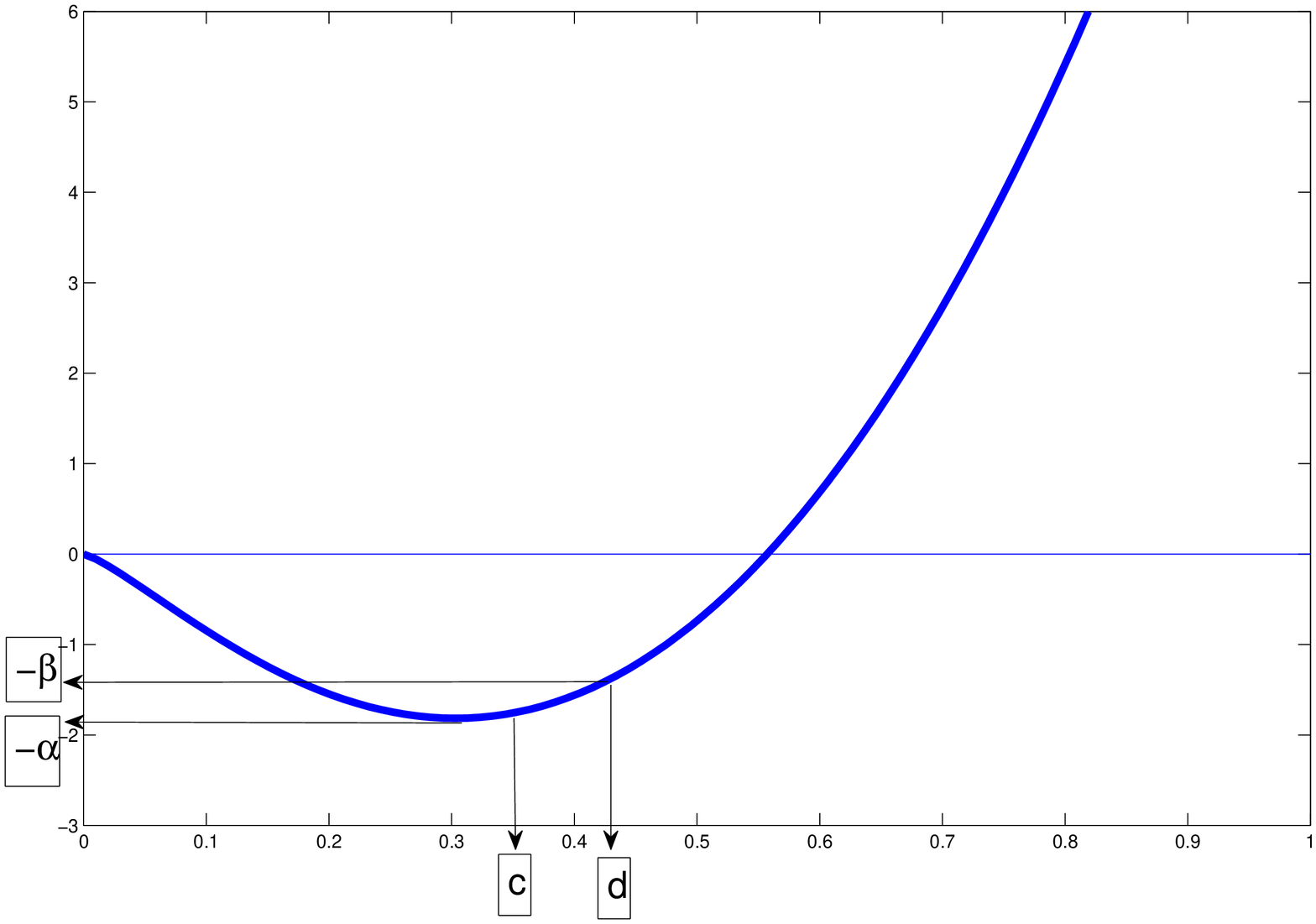} 
	\caption{Behaviour of $ \mbox{Re}\left( \phi_\I\right) $ }
	\label{psifigure}
\end{figure}
It follows from \eqref{Nchoisi} and \eqref{definitiondelta} that $||v_0||_{L^2(\R)}=\delta \leq \eta$, and thus for this initial data, we get that $||v(t,\cdot)||_{L^2(\R)} \leq \varepsilon$ for all $t>0$.  \\
For $n  \in \left\lbrace 0, \cdots, N \right\rbrace $, we have
\begin{eqnarray*}
v(nt_0) &=& S(nt_0)v_0 + \sum_{k=0}^{n-1} S\left( (n-1-k)t_0\right)\left[v((k+1)t_0)-S(t_0)v(kt_0)\right], \\
&=&   S(nt_0)v_0 + \sum_{k=0}^{n-1} S\left( (n-1-k)t_0\right)\left[T(t_0)v(kt_0)-S(t_0)v(kt_0)\right].
\end{eqnarray*}
Hence, we get using \eqref{difft0} and the $L^2$-estimate given in Remark \ref{remarkL2}
\begin{eqnarray*}
||v(nt_0) - S(nt_0)v_0||_{L^2(\R)} &\leq & \sum_{k=0}^{n-1} e^{\alpha t_0(n-k-1)} ||v((k+1)t_0)-S(t_0)v(kt_0)||_{L^2(\R)}, \\
&\leq &  \sum_{k=0}^{n-1} e^{\alpha t_0(n-k-1)} b_0||v(kt_0)||^2_{L^2(\R)}, \\
&\leq & \sum_{k=0}^{n-1} e^{\alpha t_0(n-k-1)} b_0 e^{2\alpha k t_0} ||v_0||_{L^2(\R)}^2, \\
&=& \delta^2 b_0 e^{\alpha t_0 (n-1)} \frac{e^{\alpha n t_0}-1}{e^{\alpha t_0}-1}, \\
&\leq& \frac{\delta}{4} e^{\alpha t_0 n}.  
\end{eqnarray*}
Moreover, by Plancherel formula, we have
\begin{eqnarray*}
\n S(t)v_0\n_{\LdeuxR}^2 = \n K(t,\cdot) \ast v_0 \n_{\LdeuxR}^2 &=&    \n \F(K(t,\cdot) \ast v_0) \n_{\LdeuxR}^2, \\
&= & \int_\R |\F(K(t,\cdot))(\xi) \F(v_0)(\xi)|^2 d\xi, \\
&=&  \int_c^d \frac{\delta^2 }{d-c} e^{-2t \, \mbox{Re} (\phi_{\I})(\xi)} d\xi, \\
&\geq& e^{2 \beta t}  \n  v_0 \n_{\LdeuxR}^2.
\end{eqnarray*}
We finally infer that 
\begin{eqnarray*}
||v(Nt_0)||_{L^2(\R)} &\geq& ||S(Nt_0)v_0||_{L^2(\R)} - ||v(Nt_0) - S(Nt_0)v_0||_{L^2(\R)}, \\ 
&\geq & \delta e^{\beta N t_0} - \frac{\delta}{4} e^{\alpha t_0 N}, \\
&\geq& \frac{e^{\alpha t_0}-1}{16 b_0} > \varepsilon, 
\end{eqnarray*}
which gives us a contradiction and completes the proof of this theorem.

\end{proof}

\begin{remark}
We can give a physical interpretation of this result: a flat profile $u_\phi=constant$ is unstable under the morphodynamics described by the Fowler model.
\end{remark}

\section{Numerical simulations\label{numerik}}


The spatial discretization is given by a set of points ${x_{j}; j=1,...,N}$ and the discretization in time is represented by a sequence of times $t^0=0<...<t^n<...<T$. For the sake of simplicity we will assume constant step size $\delta x$ and $\delta t$ in space and time, respectively. The discrete solution at a point will be represented by $u^n_j \approx u(t^n,x_j)$. 
The schemes consist in computing approximate values $u^n_j$ of solution to \eqref{fowlereqn} on $[n\delta t, (n+1) \delta t[\times [j \delta x, (j+1) \delta x[$ for $n \in \mathbb{N}$ and $j \in \mathbb{N}$ thanks to the following relation:
\begin{equation}
\label{FDscheme}
\frac{u_{j}^{n+1}-u_{j}^{n}}{\delta t}+ \frac{1}{2} \frac{(u_{j}^{n})^2-(u_{j-1}^{n})^2}{\delta x}- \frac{u_{j+1}^{n}-2u_{j}^{n}+u_{j-1}^{n}}{\delta x^{2}}+ \I_{\delta x}[u^{n}]_{j}=0,
\end{equation} 
where $\I_{\delta x}$ is the discretization of the nonlocal term $\I$. 
We use a basic quadrature rule to approximate $\I$ given by \eqref{nonlocalterm}
\begin{equation} \label{discretization14}
\I_{\delta x}[\varphi]_{j}=\delta x^{-4/3} \sum^{+ \infty}_{l=1} l^{-1/3} \left(\varphi_{j-l+1}-2\varphi_{j-l}+ \varphi_{j-l-1}\right).
\end{equation}
Let us note that for compactly supported initial datum the sum is a finite sum.
In this section, we want to simulate the instability stated previously. To this aim we must be careful to distinguish between the instability which stems from the model and numerical instabilities, which could be caused by a careless discretization.   \\
The numerical stability ensures that the difference between the approximate solution and the exact solution remains bounded for increasing $t$ for $\delta x, \delta t$ given. \\

As noticed, the Fowler model amplifies  slowly the low frequencies whereas the high frequencies are quickly dampened. Thus, the classical notion of $A$-stability (strong stability) and $C$-stability are not suitable nor desirable. A new definition of stability for this model has been considered in \cite{azaf} and numerical stability criteria have been obtained.\\

In the following numerical test, we have to take care to choose $\delta x$ and $\delta t$ accurately following the stability condition, see \cite{azaf}. We expose in Figure \ref{rides} the evolution of an initial flat bottom disturbed by a small bump for different times. As we can see the perturbation creates small ripples that grow with time without bound. This result proves that the perturbed solution goes away from the non-perturbed solution. This
illustrates the instability of constant solutions for the Fowler model \eqref{fowlereqn}. Figure \ref{logevo} shows the log-plot for the $L^2$-norm evolution of this perturbed bottom, which confirms the validity of the estimate given in Remark \ref{remarkL2}.

\begin{figure}[h!]
	\centering
        \subfigure[ Evolution of the solution disturbed. ]
	{\includegraphics[scale = 0.2]{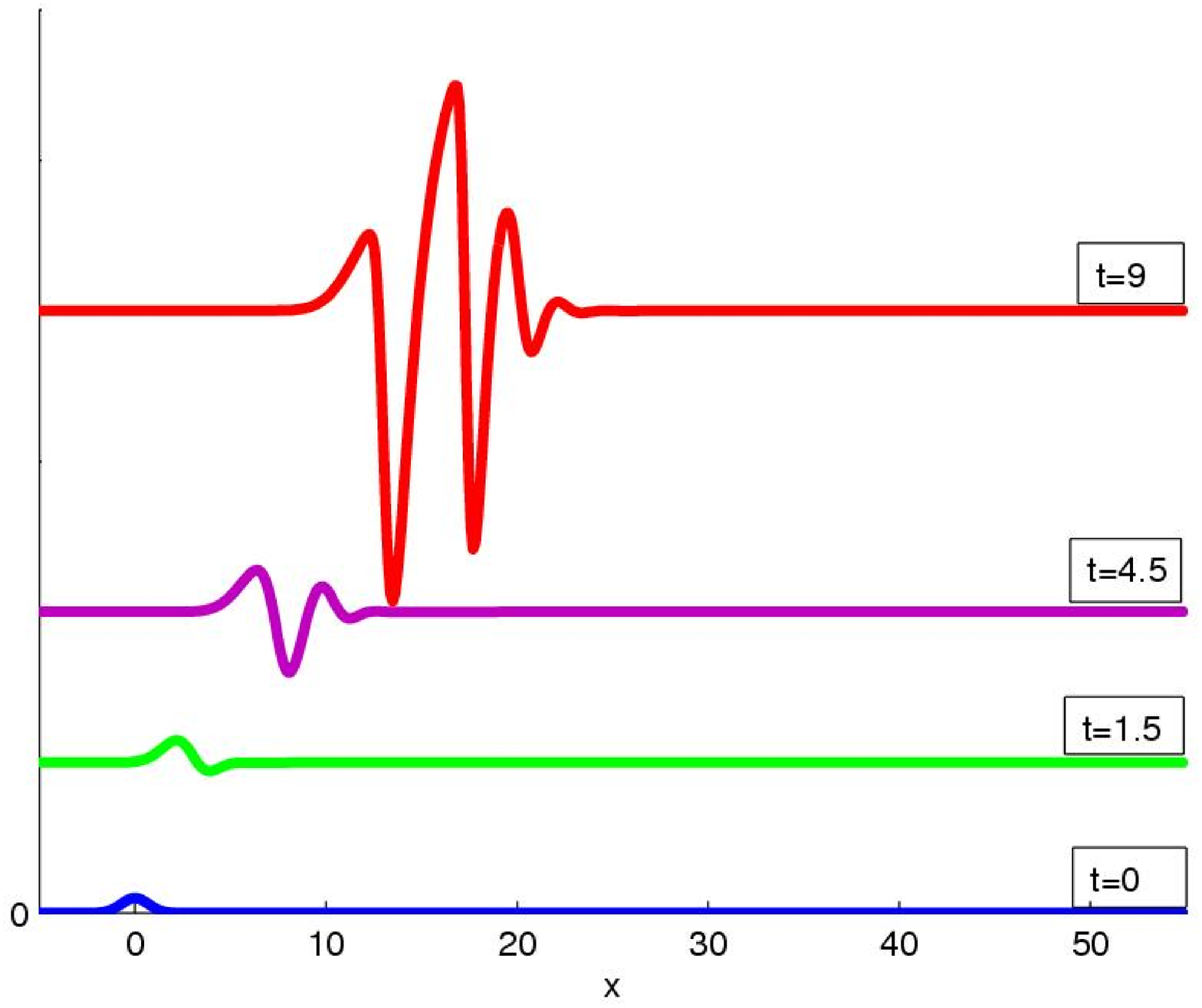} \label{rides} }
	\subfigure[ Evolution of $\log( ||v(t, \cdot)||_{L^2})$ ]
	{\includegraphics[scale=0.4]{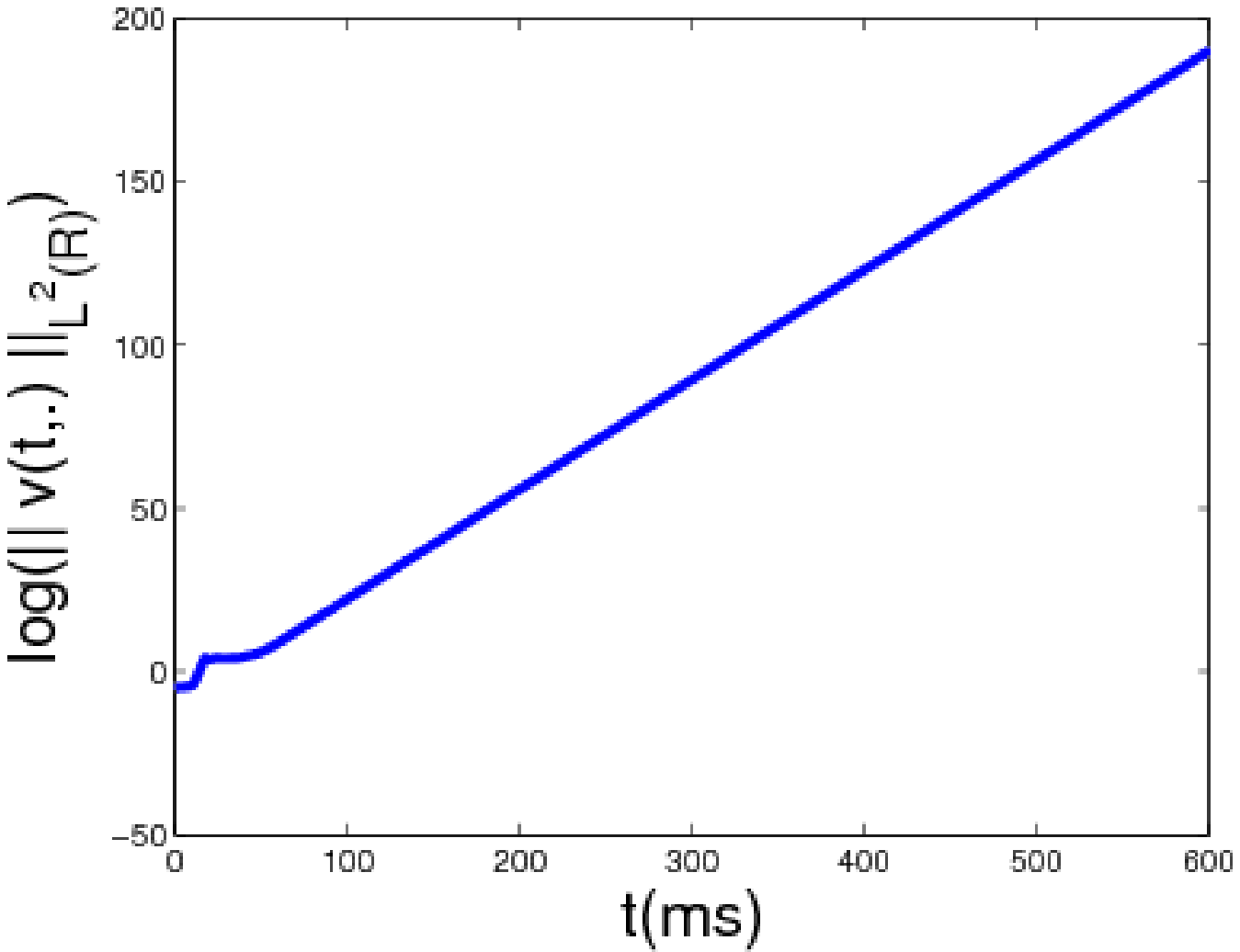} \label{logevo} }
        \caption{ Evolution of a flat bottom disturbed by a small bump. \label{instability}
}
\end{figure}

\section{Acknowledgements} 
We thank Pascal Azerad and Cl\'ement Gallo for helpful comments. The author is supported by the ANR MATHOCEAN ANR-08-BLAN-0301-02

\end{document}